\documentclass[11pt,headings=small]{scrartcl}

\usepackage[sumlimits]{amsmath}
\usepackage{amssymb,amsfonts,amsthm}
\usepackage[english]{babel}
\usepackage[latin1]{inputenc}
\usepackage[T1]{fontenc}
\usepackage{lmodern}
\usepackage{graphicx}
\usepackage{scrpage2}
\usepackage{hyperref}
\usepackage{setspace}
\usepackage{textcomp}
\usepackage{enumerate}
\usepackage{xfrac}
\usepackage{helvet}
\usepackage{bbm}
\usepackage{multirow}
\usepackage{bigstrut}
\usepackage{booktabs}

\newtheorem{thm}{Theorem}[section]

\newtheorem{cor}[thm]{Corollary}
\newtheorem{lem}[thm]{Lemma}
\theoremstyle{definition}

\def\P{\mathcal{P}}
\def\R{\mathbb{R}}
\def\Cx{\mathbb{C}}
\def\N{\mathbb{N}}

\def\F{\mathbb{F}}

\def\D{\mathbb{D}}
\def\B{\mathbb{B}}
\def\Dn{\mathfrak{D}}
\def\Cn{\mathfrak{C}}
\def\1{\mathbbm{1}}

\newcommand\dint{{\rm d}}
\newcommand{\im}{{\rm i}}
\DeclareMathOperator{\rank}{rank}
\DeclareMathOperator{\wal}{wal}

\DeclareMathOperator{\e}{e}

\allowdisplaybreaks[1]
\clearscrheadfoot
\ihead{\headmark}
\ohead{\thepage}

\begin{document}
\pagestyle{scrheadings}
\onehalfspacing

\title{$L_p$- and $S_{p,q}^rB$-discrepancy of (order $2$) digital nets}
\author{Lev Markhasin \\ \scriptsize Institut f\"ur Stochastik und Anwendungen, Universit\"at Stuttgart \\ \scriptsize Pfaffenwaldring 57, 70569 Stuttgart, Germany \\ \scriptsize email: lev.markhasin@mathematik.uni-stuttgart.de}
\maketitle

\begin{abstract}
Dick proved that all dyadic order $2$ digital nets satisfy optimal upper bounds on the $L_p$-discrepancy. We prove this for arbitrary prime base $b$ with an alternative technique using Haar bases. Furthermore, we prove that all digital nets satisfy optimal upper bounds on the discrepancy function in Besov spaces with dominating mixed smoothness for a certain parameter range and enlarge that range for order $2$ digitals nets. The discrepancy function in Triebel-Lizorkin and Sobolev spaces with dominating mixed smoothness is considered as well.
\end{abstract}

\noindent{\footnotesize {\it 2010 Mathematics Subject Classification.} Primary 11K06,11K38,42C10,46E35,65C05. } \\[5mm]
\textit{Acknowledgement:} The author would like to thank the anonymous referee for his/her helpful comments.

\section{Introduction and results}
Let $N$ be a positive integer and let $\P$ be a point set in the unit cube $[0,1)^d$ with $N$ points. Then the discrepancy function $D_{\P}$ is defined as
\begin{align}
D_{\P}(x) = \frac{1}{N} \sum_{z \in \P} \chi_{[0,x)}(z) - x_1 \cdots x_d
\end{align}
for any $x = (x_1, \ldots, x_d) \in [0,1)^d$. By $\chi_{[0,x)}$ we mean the characteristic function of the interval $[0,x) = [0,x_1)\times\ldots\times[0,x_d)$, so the term $\sum_z \chi_{[0,x)}(z)$ is equal to the number of points of $\P$ in $[0,x)$. This means that $D_{\P}$ measures the deviation of the number of points of $\P$ in $[0,x)$ from the fair number of points $N |[0,x)| = N \, x_1 \cdots x_d$, which would be achieved by a (practically impossible) perfectly uniform distribution of the points of $\P$.

Usually one is interested in calculating the norm of the discrepancy function in some normed space of functions on $[0,1)^d$ to which the discrepancy function belongs. A well known result concerns $L_p([0,1)^d)$-spaces for $1<p<\infty$. There exists a constant $c_{p,d} > 0$ such that for every positive integer $N$ and all point sets $\P$ in $[0,1)^d$ with $N$ points, we have
\begin{align}
\left\|D_{\P}|L_p([0,1)^d)\right\|\geq c_{p,d}\,\frac{\left(\log N\right)^{(d-1)/2}}{N}.
\end{align}
This was proved by Roth \cite{R54} for $p=2$ and by Schmidt \cite{S77} for arbitrary $1<p<\infty$. The currently best known value for $c_{2,d}$ can be found in \cite{HM11}. Furthermore, there exists a constant $C_{p,d} > 0$ such that for every positive integer $N$, there exists a point set $\P$ in $[0,1)^d$ with $N$ points such that
\begin{align}\label{UpperLp}
\left\|D_{\P}|L_p([0,1)^d)\right\|\leq C_{p,d}\,\frac{\left(\log N\right)^{(d-1)/2}}{N}.
\end{align}
This was proved by Davenport \cite{D56} for $p=2,d=2$, by Roth \cite{R80} for $p=2$ and arbitrary $d$ and finally by Chen \cite{C80} in the general case. The currently best known  value for $C_{2,d}$ can be found in \cite{DP10} and \cite{FPPS10}.

There are results for the $L_1([0,1)^d)$- and the star ($L_{\infty}([0,1)^d)$-) discrepancy though there are still gaps between lower and upper bounds, see \cite{H81}, \cite{S72}, \cite{BLV08}. As general references for studies of the discrepancy function we refer to the monographs \cite{DP10}, \cite{NW10}, \cite{M99}, \cite{KN74} and surveys \cite{B11}, \cite{Hi14}, \cite{M13c}. The problem of point disribution is closely related to numerical integration, we refer to \cite[Chapter 2]{KN74} and \cite[Section 2.4]{DP10} for more on this subject.

Roth's and Chen's original proofs of \eqref{UpperLp} were probabilistic. Explicit constructions of point sets with good $L_p$-discrepancy in arbitrary dimension have not been known for a long time. Chen and Skriganov \cite{CS02} (see also \cite{CS08} and \cite{DP10}) gave explicit constructions satisfying the optimal bound on the $L_2$-discrepancy and Skriganov \cite{S06} later gave explicit constructions satisfying the optimal bound on the $L_p$-discrepancy. The constructions of Chen and Skriganov are digital nets over $\F_b$ with large Hamming weight. Dick and Pillichshammer \cite{DP14a} gave alternative constructions. Their constructions are order $3$ digital nets over $\F_2$. They also constructed digital sequences with optimal bounds on the $L_2$-discrepancy. Dick \cite{D14} gave further constructions which are order $2$ digital nets over $\F_2$. Here we generalize Dick's approach to order $2$ digital nets over $\F_b$ for every prime number $b$, which is stated in the following result.

\begin{thm}\label{main_thm_L2}
 There exists a constant $C_{d,b,v} > 0$ such that for every positive integer $n$ and every order $2$ digital $(v,n,d)$-net $\P_n^b$ over $\F_b$ we have
\[ \left\|D_{\P_n^b}|L_2([0,1)^d)\right\|\leq C_{d,b,v}\,\frac{n^{(d-1)/2}}{b^n}. \]
\end{thm}

Our proof uses an alternative technique to Chen and Skriganov and Dick and Pillichhammer relying on Haar bases.

Furthermore, there are results for the discrepancy in other function spaces, like Hardy spaces, logarithmic and exponential Orlicz spaces, weighted $L_p$-spaces and BMO (see \cite{B11} for results and further literature). 

Here, we are interested in Besov ($S_{p,q}^rB([0,1)^d)$), Triebel-Lizorkin ($S_{p,q}^rF([0,1)^d)$) and Sobolev ($S_p^rH([0,1)^d)$) spaces with dominating mixed smoothness. Triebel \cite{T10} proved that for all $1\leq p,q\leq\infty$ with $q<\infty$ if $p=\infty$ and all $r\in\R$ satisfying $1/p-1<r<1/p$, there exists a constant $c_{p,q,r,d}>0$ such that for every integer $N \geq 2$ and all point sets $\P$ in $[0,1)^d$ with $N$ points, we have
\begin{align}
\left\|D_{\P}|S_{p,q}^rB([0,1)^d)\right\|\geq c_{p,q,r,d}\,N^{r-1}\,\left(\log N\right)^{(d-1)/q}.
\end{align}
With the additional condition that $q>1$ if $p=\infty$ there exists a constant $C_{p,q,r,d}>0$ such that for every positive integer $N$, there exists a point set $\P$ in $[0,1)^d$ with $N$ points such that
\[ \left\|D_{\P}|S_{p,q}^rB([0,1)^d)\right\|\leq C_{p,q,r,d}\,N^{r-1}\,\left(\log N\right)^{(d-1)(1/q+1-r)}. \]
Hinrichs \cite{Hi10} proved for $d=2$ that for all $1\leq p,q\leq\infty$ and all $0\leq r<1/p$ there exists a constant $C_{p,q,r}>0$ such that for every integer $N \geq 2$ there exists a point set $\P$ in $[0,1)^2$ with $N$ points such that
\[ \left\|D_{\P}| S_{p,q}^rB([0,1)^2)\right\|\leq C_{p,q,r}\,N^{r-1}\,\left(\log N\right)^{1/q}. \]
Markhasin \cite{M13b} proved that for all $1\leq p,q\leq\infty$ and all $0<r<1/p$ there exists a constant $C_{p,q,r,d}>0$ such that for every integer $N \geq 2$ there exists a point set $\P$ in $[0,1)^d$ with $N$ points such that
\begin{align}
\left\|D_{\P}| S_{p,q}^rB([0,1)^d)\right\|\leq C_{p,q,r,d}\,N^{r-1}\,\left(\log N\right)^{(d-1)/q}.
\end{align}
\cite{M13b} relied for the proof on explicit constructions. It was shown that the already mentioned constructions by Chen and Skriganov additionally have optimal bounds on the $S_{p,q}^rB$-discrepancy. The notion $S_{p,q}^rB$-discrepancy will be defined in the next section. In $d=2$ also (generalized) Hammersley point sets can be used (see \cite{Hi10}, \cite{M13a}). Our goal is to prove that there are also other point sets with optimal bounds on the $S_{p,q}^rB$-discrepancy. Furthermore we prove results for the spaces $S_{p,q}^r F([0,1)^d)$ and $S_p^r H([0,1)^d)$.

\begin{thm}\label{main_thm_SpqrB_ord1}
 Let $1\leq p<\infty$, $1\leq q\leq\infty$ and $0<r<1/p$. There exists a constant $C_{p,q,r,d,b,v} > 0$ such that for every integer $n$ and every order $1$ digital $(v,n,d)$-net $\P_n^b$ over $\F_b$ we have
\[ \left\|D_{\P_n^b}|S_{p,q}^r B([0,1)^d)\right\|\leq C_{p,q,r,d,b,v}\,b^{n(r-1)}\,n^{(d-1)/q}. \]
\end{thm}

\begin{thm}\label{main_thm_SpqrB_ord2}
 Let $1\leq p,q\leq\infty$, ($q>1$ if $p=\infty$) and $0\leq r<1/p$. There exists a constant $C_{p,q,r,d,b,v} > 0$ such that for every positive integer $n$ and every order $2$ digital $(v,n,d)$-net $\P_n^b$ over $\F_b$ we have
\[ \left\|D_{\P_n^b}|S_{p,q}^r B([0,1)^d)\right\|\leq C_{p,q,r,d,b,v}\,b^{n(r-1)}\,n^{(d-1)/q}. \]
\end{thm}

Applying embeddings between Besov and Triebel-Lizorkin spaces that we will state later we obtain the following results.

\begin{cor}\label{main_cor_SpqrF_ord1}
 Let $1\leq p,q<\infty$ and $0<r<1/\max(p,q)$. There exists a constant $C_{p,q,r,d,b,v} > 0$ such that for every positive integer $n$ and every order $1$ digital $(v,n,d)$-net $\P_n^b$ over $\F_b$ we have
\[ \left\|D_{\P_n^b}|S_{p,q}^r F([0,1)^d)\right\|\leq C_{p,q,r,d,b,v}\,b^{n(r-1)}\,n^{(d-1)/q}. \]
\end{cor}

\begin{cor}\label{main_cor_SpqrF_ord2}
 Let $1\leq p,q<\infty$ and $0\leq r<1/\max(p,q)$. There exists a constant $C_{p,q,r,d,b,v} > 0$ such that for every positive integer $n$ and every order $2$ digital $(v,n,d)$-net $\P_n^b$ over $\F_b$ we have
\[ \left\|D_{\P_n^b}|S_{p,q}^r F([0,1)^d)\right\|\leq C_{p,q,r,d,b,v}\,b^{n(r-1)}\,n^{(d-1)/q}. \]
\end{cor}

The following results are just special cases of the last corollaries.

\begin{cor}\label{main_cor_SprH_ord1}
 Let $1\leq p<\infty$ and $0<r<1/\max(p,2)$. There exists a constant $C_{p,r,d,b,v} > 0$ such that for every positive integer $n$ and every order $1$ digital $(v,n,d)$-net $\P_n^b$ over $\F_b$ we have
\[ \left\|D_{\P_n^b}|S_p^r H([0,1)^d)\right\|\leq C_{p,r,d,b,v}\,b^{n(r-1)}\,n^{(d-1)/2}. \]
\end{cor}

\begin{cor}\label{main_cor_SprH_ord2}
 Let $1\leq p<\infty$ and $0\leq r<1/\max(p,2)$. There exists a constant $C_{p,r,d,b,v} > 0$ such that for every positive integer $n$ and every order $2$ digital $(v,n,d)$-net $\P_n^b$ over $\F_b$ we have
\[ \left\|D_{\P_n^b}|S_p^r H([0,1)^d)\right\|\leq C_{p,r,d,b,v}\,b^{n(r-1)}\,n^{(d-1)/2}. \]
\end{cor}

\begin{cor}\label{main_thm_Lp}
 Let $1\leq p<\infty$. There exists a constant $C_{p,d,b,v} > 0$ such that for every positive integer $n$ and every order $2$ digital $(v,n,d)$-net $\P_n^b$ over $\F_b$ we have
\[ \left\|D_{\P_n^b}|L_p([0,1)^d)\right\|\leq C_{p,d,b,v}\,\frac{n^{(d-1)/2}}{b^n}. \]
\end{cor}

The difference in the results of Theorem \ref{main_thm_SpqrB_ord1} and Theorem \ref{main_thm_SpqrB_ord2} seems to be small. But the point is that an order $2$ digital net is also an order $1$ digital net, so assuming a stronger condition we enlarge the range of the parameter $r$, namely adding the case $r = 0$, which is essential to obtain results for $L_p$-spaces.

We state the results with implicit constants depending on $v$ though we get this dependence explicitly. The readers interested in the $v$-dependency can find it in the proofs of the theorems, namely \eqref{main_thm_L2_v_dep}, \eqref{main_thm_SpqrB_ord1_v_dep} and \eqref{main_thm_SpqrB_ord2_v_dep}.

We point out that obviously Theorem \ref{main_thm_L2} is a consequence of Corollary \ref{main_thm_Lp}. Nevertheless, we will prove them independently, so that readers without a background in function spaces with dominating mixed smoothness (which is required for the proof of Corollary~\ref{main_thm_Lp}) will be able to understand the proof of the $L_2$ bound.

Theorems \ref{main_thm_SpqrB_ord1} and \ref{main_thm_SpqrB_ord2} are consistent with older results. The proofs in \cite{M13b} only relied on order $1$ digital $(v,n,d)$-net properties of the Chen-Skriganov point sets and not the large Hamming weight so the weeker result was obtained while (generalized) Hammersley point sets used by Hinrichs and Markhasin are order $2$ digital $(0,n,2)$-nets and yielded a stronger result.

The bounds on the discrepancy in Besov spaces is closely connected to the integration error. We refer to \cite{T10}, \cite[Chapter 5]{M13c} and \cite{U14} for more information on this connection and for error bounds in Besov, Triebel-Lizorkin and Sobolev spaces with dominating mixed smoothness.

\section{Function spaces with dominating mixed smoothness}
We define the spaces $S_{p,q}^r B([0,1)^d)$, $S_{p,q}^r F([0,1)^d)$ and $S_p^r H([0,1)^d)$ according to \cite{T10}. Let $\mathcal{S}(\R^d)$ denote the Schwartz space and $\mathcal{S}'(\R^d)$ the space of tempered distributions on $\R^d$. For $\varphi\in\mathcal{S}(\R^d)$ we denote by $\mathcal{F}\varphi$ the Fourier transform of $\varphi$ and extend it to $\mathcal{S}'(\R^d)$ in the usual way. For $f\in \mathcal{S}'(\R^d)$ the Fourier transform is given as $\mathcal{F} f(\varphi) = f(\mathcal{F}\varphi),\; \varphi\in\mathcal{S}(\R^d)$. Analogously we proceed with the inverse Fourier transform $\mathcal{F}^{-1}$.

Let $\varphi_0 \in \mathcal{S}(\R)$ satisfy $\varphi_0(x) = 1$ for $|x| \leq 1$ and $\varphi_0(x) = 0$ for $|x| > \frac{3}{2}$. Let $\varphi_k(x) = \varphi_0(2^{-k} x) - \varphi_0(2^{-k + 1} x)$ where $x \in \R, \, k \in \N$ and $\varphi_{\bar{k}}(x) = \varphi_{k_1}(x_1) \cdots \varphi_{k_d}(x_d)$ where $\bar{k} = (k_1,\ldots,k_d) \in \N_0^d$ and $x = (x_1,\ldots,x_d) \in \R^d$. The functions $\varphi_{\bar{k}}$ are a dyadic resolution of unity since
\[ \sum_{{\bar{k}} \in \N_0^d} \varphi_k(x) = 1 \]
for all $x \in \R^d$. The functions $\mathcal{F}^{-1}(\varphi_{\bar{k}} \mathcal{F} f)$ are entire analytic functions for every $f \in \mathcal{S}'(\R^d)$.

Let $0 < p,q \leq \infty$ and $r \in \R$. The Besov space with dominating mixed smoothness $S_{p,q}^r B(\R^d)$ consists of all $f \in \mathcal{S}'(\R^d)$ with finite quasi-norm
\begin{align}
\left\| f | S_{p,q}^r B(\R^d) \right\| = \left( \sum_{{\bar{k}} \in \N_0^d} 2^{r (k_1 + \ldots + k_d) q} \left\| \mathcal{F}^{-1}(\varphi_{\bar{k}} \mathcal{F} f) | L_p(\R^d) \right\|^q \right)^{\frac{1}{q}}
\end{align}
with the usual modification if $q = \infty$.

Let $0 < p < \infty$, $0 < q \leq \infty$ and $r \in \R$. The Triebel-Lizorkin space with dominating mixed smoothness $S_{p,q}^r F(\R^d)$ consists of all $f \in \mathcal{S}'(\R^d)$ with finite quasi-norm
\begin{align}
\left\| f | S_{p,q}^r F(\R^d) \right\| = \left\| \left( \sum_{{\bar{k}} \in \N_0^d} 2^{r (k_1 + \ldots + k_d) q} \left| \mathcal{F}^{-1}(\varphi_{\bar{k}} \mathcal{F} f)(\cdot) \right|^q \right)^{\frac{1}{q}} | L_p(\R^d) \right\|
\end{align}
with the usual modification if $q = \infty$.

Let $\mathcal{D}([0,1)^d)$ consist of all complex-valued infinitely differentiable functions on $\R^d$ with compact support in the interior of $[0,1)^d$ and let $\mathcal{D}'([0,1)^d)$ be its dual space of all distributions in $[0,1)^d$. The Besov space with dominating mixed smoothness $S_{p,q}^r B([0,1)^d)$ consists of all $f \in \mathcal{D}'([0,1)^d)$ with finite quasi-norm
\begin{align}
\left\| f | S_{p,q}^r B([0,1)^d) \right\| = \inf \left\{ \left\| g | S_{p,q}^r B(\R^d) \right\| : \: g \in S_{p,q}^r B(\R^d), \: g|_{[0,1)^d} = f \right\}.
\end{align}
The Triebel-Lizorkin space with dominating mixed smoothness $S_{p,q}^r F([0,1)^d)$ consists of all $f \in \mathcal{D}'([0,1)^d)$ with finite quasi-norm
\begin{align}
\left\| f | S_{p,q}^r F([0,1)^d) \right\| = \inf \left\{ \left\| g | S_{p,q}^r F(\R^d) \right\| : \: g \in S_{p,q}^r F(\R^d), \: g|_{[0,1)^d} = f \right\}.
\end{align}
The spaces $S_{p,q}^r B(\R^d), \, S_{p,q}^r F(\R^d), \, S_{p,q}^r B([0,1)^d)$ and $S_{p,q}^r F([0,1)^d)$ are quasi-Banach spaces. We define the Sobolev space with dominating mixed smoothness as
\begin{align}
 S_p^r H([0,1)^d) = S_{p,2}^r F([0,1)^d).
\end{align}
If $r \in \N_0$ then it is denoted by $S_p^r W ([0,1)^d)$ and is called classical Sobolev space with dominating mixed smoothness. An equivalent norm for $S_p^r W([0,1)^d)$ is
\[ \sum_{\alpha \in \N_0^d; \, 0 \leq \alpha_i \leq r} \left\| D^{\alpha} f | L_p([0,1)^d) \right\|. \]
Of special interest is the case $r = 0$ since
\[ S_p^0 H([0,1)^d) = L_p([0,1)^d). \]

The Besov and Triebel-Lizorkin spaces can be embedded in each other (see \cite{T10} or \cite[Corollary 1.13]{M13c}). We point out that the following embedding is a combination of well known results and might look odd at the first glance.
\begin{lem}\label{lem_emb_BF}
Let $0<p,q<\infty$ and $r\in\R$. Then we have
\[ S_{\max(p,q),q}^r B([0,1)^d)\hookrightarrow S_{p,q}^r F([0,1)^d)\hookrightarrow S_{\min(p,q),q}^r B([0,1)^d). \]
\end{lem}

The reader interested in function spaces is referred to \cite{H10}, \cite{ST87} and \cite{T10} and the references given there.

A goal of this paper is to analyze the discrepancy function in spaces $S_{p,q}^r B([0,1)^d)$, $S_{p,q}^r F([0,1)^d)$ and $S_{p}^r H([0,1)^d)$. We define $S_{p,q}^r B([0,1)^d)$-discrepancy as
\[ \inf_{\P} \left\| D_{\P} | S_{p,q}^r B([0,1)^d) \right\| \]
where the infimum is taken over all point sets with $N$ points. Analogously we define $S_{p,q}^r F([0,1)^d)$-discrepancy and $S_p^r H([0,1)^d)$-discrepancy.

\section{Haar and Walsh bases}
We write $\N_{-1}=\N_0\cup\{-1\}$. Let $b\geq 2$ be an integer. We write $\D_j = \{0,1,\ldots, b^j-1\}$ and $\B_j = \{1,\ldots,b-1\}$ for $j \in \N_0$ and $\D_{-1} = \{0\}$ and $\B_{-1} = \{1\}$. For $j = (j_1,\dots,j_d)\in\N_{-1}^d$ let $\D_j = \D_{j_1}\times\ldots\times\D_{j_d}$ and $\B_j = \B_{j_1}\times\ldots\times\B_{j_d}$. For a real number $a$ we write $a_+ = \max(a,0)$ and for $j\in\N_{-1}^d$ we write $|j|_+ = j_1{}_+ + \ldots + j_d{}_+$.

For $j \in \N_0$ and $m \in \D_j$ we call the interval
\[I_{j,m} = \big[ b^{-j} m, b^{-j} (m+1) \big) \]
the $m$-th $b$-adic interval in $[0,1)$ on level $j$. We put $I_{-1,0}=[0,1)$ and call it the $0$-th $b$-adic interval in $[0,1)$ on level $-1$. For any $k = 0,\ldots,b - 1$ let $I_{j,m}^k = I_{j + 1,bm + k}$. We put $I_{-1,0}^{-1} = I_{-1,0} = [0,1)$. For $j \in \N_{-1}^d$ and $m = (m_1, \ldots, m_d) \in \D_j$ we call
\[ I_{j,m} = I_{j_1,m_1} \times \ldots \times I_{j_d,m_d}\]
the $m$-th $b$-adic interval in $[0,1)^d$ on level $j$. We call the number $|j|_+$ the order of the $b$-adic interval $I_{j,m}$. Its volume is $b^{-|j|_+}$.

Let $j \in \N_{0}$, $m \in \D_j$ and $l \in \B_j$. Let $h_{j,m,l}$ be the function on $[0,1)$ with support in $I_{j,m}$ and the constant value $\e^{\frac{2\pi \im}{b} l k}$ on $I_{j,m}^k$ for any $k=0, \ldots, b - 1$. We put $h_{-1,0,1} = \chi_{I_{-1,0}}$ on $[0,1)$ which is the characteristic function of the interval $I_{-1,0}$.

Let $j \in \N_{-1}^d$, $m \in \D_j$ and $l = (l_1, \ldots, l_d) \in \B_j$. The function $h_{j,m,l}$ given as the tensor product
\[ h_{j,m,l}(x) = h_{j_1,m_1,l_1}(x_1) \cdots h_{j_d,m_d,l_d}(x_d) \]
for $x = (x_1, \ldots, x_d) \in [0,1)^d$ is called a $b$-adic Haar function on $[0,1)^d$. The set of functions $\{h_{j,m,l}: \, j \in \N_{-1}^d, \, m \in \D_j, \, l\in\B_j\}$ is called $b$-adic Haar basis on $[0,1)^d$. We can use the Haar basis for calculating the norms of the discrepancy function.

The following result is \cite[Theorem 2.1]{M13c} and is a tool for calculating the $L_2$-discrepancy.
\begin{thm}\label{L2Haar}
The system
\[ \left\{b^{\frac{|j|_+}{2}}h_{j,m,l} \,:\,j\in\N_{-1}^d,\,m\in\D_j,\,l\in\B_j\right\} \]
is an orthonormal basis of $L_2([0,1)^d)$, an unconditional basis of $L_p([0,1)^d)$ for $1 < p < \infty$ and a conditional basis of $L_1([0,1)^d)$. For any function $f\in L_2([0,1)^d)$ we have
\[ \left\|f|L_2([0,1)^d)\right\|^2 = \sum_{j \in \N_{-1}^d} b^{|j|} \sum_{m\in\D_j,\,l\in\B_j}|\langle f,h_{j,m,l}\rangle|^2. \]
\end{thm}

The following result is \cite[Theorem 2.11]{M13c} and is a tool for calculating the $S_{p,q}^r B$-discrepancy.
\begin{thm}\label{SpqrBHaar}
Let $0<p,q\leq\infty$, ($q>1$ if $p=\infty$) and $1/p-1<r<\min(1/p,1)$. Let $f\in\mathcal{D}'([0,1)^d)$. Then $f\in S_{p,q}^r B([0,1)^d)$ if and only if it can be represented as
\begin{align}\label{repres_f}
f = \sum_{j\in\N_{-1}^d} b^{|j|_+}\sum_{m\in\D_j,\,l\in\B_j}\mu_{j,m,l}\,h_{j,m,l}
\end{align}
for some sequence $(\mu_{j,m,l})$ satisfying
\begin{align} \label{eq_quasinorm}
\left(\sum_{j\in\N_{-1}^d} b^{|j|_+(r-1/p+1)q}\left(\sum_{m\in\D_j,\,l\in\B_j}|\mu_{j,m,l}|^p\right)^{q/p}\right)^{1/q}<\infty.
\end{align}
The convergence of \eqref{repres_f} is unconditional in $\mathcal{D}'([0,1)^d)$ and in any $S_{p,q}^\rho B([0,1)^d)$ with $\rho < r$. The representation \eqref{repres_f} of $f$ is unique with the $b$-adic Haar coefficients $\mu_{j,m,l} = \langle f,h_{j,m,l}\rangle$.
The expression \eqref{eq_quasinorm} is an equivalent quasi-norm in $S_{p,q}^r B([0,1)^d)$.
\end{thm}

A weight from \cite{D07} will be useful for our purpose. For $\alpha \in \N$ with $b$-adic expansion $\alpha = \beta_{a_1-1} b^{a_1-1} + \ldots + \beta_{a_\nu-1} b^{a_\nu-1}$ with $0<a_1<a_2<\ldots<a_\nu$ and digits $\beta_{a_1-1},\ldots,\beta_{a_\nu-1} \in \{1,\ldots,b-1\}$, the weight of order $\sigma\in\N$ is given by
\[\varrho_\sigma(\alpha) = a_\nu+a_{\nu-1}+\ldots+a_{\max(\nu-\sigma+1,1)}. \]
Furthermore, $\varrho_\sigma(0) = 0$. It is a generalization of $\varrho_1$, first introduced in \cite{N87}.

For $\alpha = (\alpha_1, \ldots, \alpha_d) \in \N_0^d$, the weight of order $\sigma$ is given by
\[ \varrho_\sigma(\alpha) = \varrho_\sigma(\alpha_1)+\ldots+\varrho_\sigma(\alpha_d). \]

Let $\alpha\in\N$. The $\alpha$-th $b$-adic Walsh function $\wal_{\alpha}: \, [0,1) \rightarrow \Cx$ is given by
\[ \wal_{\alpha}(x) = \e^{\frac{2 \pi \im}{b} \left( \beta_{a_1-1} x_{a_1} + \ldots + \beta_{a_\nu - 1} x_{a_\nu} \right)} \]
for $x \in [0,1)$ with $b$-adic expansion  $x = x_1 b^{-1} + x_2 b^{-2} + \ldots$. Furthermore, $\wal_0=\chi_{[0,1)}$.

Let $\alpha=(\alpha_1,\ldots,\alpha_d)\in\N_0^d$. Then the $\alpha$-th $b$-adic Walsh function $\wal_{\alpha}$ on $[0,1)^d$ is given as the tensor product
\[ \wal_{\alpha}(x) = \wal_{\alpha_1}(x^1) \cdots \wal_{\alpha_d}(x^d) \]
for $x = (x^1, \ldots, x^d) \in [0,1)^d$ where by $x^i$ we mean the coordinates of $x$. The set of functions $\{\wal_{\alpha}: \, \alpha \in \N_0^d\}$ is called $b$-adic Walsh basis on $[0,1)^d$.

The $b$-adic Walsh function $\wal_{\alpha}$ is constant on $b$-adic intervals $I_{(\varrho_1(\alpha_1),\ldots,(\varrho_1(\alpha_d)),m}$ for every $m \in \D_{(\varrho_1(\alpha_1),\ldots,(\varrho_1(\alpha_d))}$. The following result is \cite[Theorem A.11]{DP10}.
\begin{lem}\label{Walsh_basis}
The system
\[ \left\{ \wal_{\alpha} \, : \, \alpha \in \N_0^d \right\} \]
is an orthonormal basis of $L_2([0,1)^d)$.
\end{lem}

\section{Digital $(v,n,d)$-nets}
Digital nets go back to Niederreiter \cite{N87}. We also refer to \cite{NP01} and \cite{DP10}. Here we quote the more general order $\sigma$ digital nets first introduced in \cite{D07} and \cite{D08}, see also \cite{DP14a}, \cite{DP14b} and \cite{D14}. In the case where $\sigma = 1$ Niederreiter's original definition is obtained.

We quote from \cite[Definitions 4.1, 4.3]{D08} to describe the digital construction method and properties of the resulting digital nets.

For a prime number $b$ let $\F_b$ denote the finite field of order $b$ identified with the set $\{0,1,\ldots,b-1\}$ equipped with arithmetic operations modulo $b$. For $s,n\in\N$ with $s\geq n$ let $C_1,\ldots,C_d$ be $s\times n$ matrices with entries from $\F_b$. For $\nu\in\{0,1,\ldots,b^n-1\}$ with the $b$-adic expansion $\nu = \nu_0 + \nu_1 b + \ldots + \nu_{n-1} b^{n-1}$ with digits $\nu_0,\nu_1,\ldots,\nu_{n-1}\in\{0,1,\ldots,b-1\}$ the $b$-adic digit vector $\bar{\nu}$ is given as $\bar{\nu} = (\nu_0,\nu_1,\ldots,\nu_{n-1})^{\top}\in\F_b^n$. Then we compute $C_i\bar{\nu}=(x_{i,\nu,1},x_{i,\nu,2},\ldots,x_{i,\nu,s})^{\top}\in\F_b^s$ for $1\leq i\leq d$. Finally we define
\[ x_{i,\nu}=x_{i,\nu,1} b^{-1}+x_{i,\nu,2} b^{-2}+\ldots+x_{i,\nu,s} b^{-s} \in[0,1) \]
and $x_{\nu}=(x_{1,\nu},\ldots,x_{d,\nu})$. We call the point set $\P_n^b=\{x_0,x_1,\ldots,x_{b^n-1}\}$ a digital net over $\F_b$.

Now let $\sigma\in\N$ and suppose $s\geq \sigma n$. Let $0\leq v\leq \sigma n$ be an integer. For every $1\leq i\leq d$ we write $C_i=(c_{i,1},\ldots,c_{i,s})^{\top}$ where $c_{i,1},\ldots,c_{i,s}\in\F_b^n$ are the row vectors of $C_i$. If for all $1\leq \lambda_{i,1}<\ldots<\lambda_{i,\eta_i}\leq s,\,1\leq i\leq d$ with
\[ \lambda_{1,1}+\ldots+\lambda_{1,\min(\eta_1,\sigma)}+\ldots+\lambda_{d,1}+\ldots+\lambda_{d,\min(\eta_d,\sigma)}\leq\sigma n - v \]
the vectors $c_{1,\lambda_{1,1}},\ldots,c_{1,\lambda_{1,\eta_1}},\ldots,c_{d,\lambda_{d,1}},\ldots,c_{d,\lambda_{d,\eta_d}}$ are linearly independet over $\F_b$, then $\P_n^b$ is called an order $\sigma$ digital $(v,n,d)$-net over $\F_b$.

The following result is \cite[Theorem 3.3]{D07}.

\begin{lem}\label{lem_digital_nets_sigma_v}
\mbox{}
\begin{enumerate}[(i)]
 \item Let $v<\sigma n$. Then every order $\sigma$ digital $(v,n,d)$-net over $\F_b$ is an order $\sigma$ digital $(v+1,n,d)$-net over $\F_b$. In particular every point set $\P_n^b$ constructed with the digital method is at least an order $\sigma$ digital $(\sigma n,n,d)$-net over $\F_b$.
 \item Let $1\leq \sigma_1\leq\sigma_2$. Then every order $\sigma_2$ digital $(v,n,d)$-net over $\F_b$ is an oder $\sigma_1$ digital $(\lceil v\sigma_1/\sigma_2 \rceil,n,d)$-net over $\F_b$.
\end{enumerate}
\end{lem}

Considering this we obtain the following geometric property going back to Niederreiter \cite{N87}.

\begin{lem}\label{nets_distribution}
 Let $\P_n^b$ be an order $\sigma$ digital $(v,n,d)$-net over $\F_b$ then every $b$-adic interval of order $n-v$ contains exactly $b^v$ points of $\P_n^b$.
\end{lem}

Let $t\in\N_0$ with $b$-adic expansion $t= \tau_0+\tau_1 b+\tau_2 b^2+\ldots$. We denote $\vec{0} = (0,\ldots,0)\in\F_b^n$. We put $\bar{t}=(\tau_0,\tau_1,\ldots,\tau_{s-1})^{\top}\in\F_b^s$ and define
\begin{align*}
 \Dn(\Cn)=\left\{t=(t_1,\ldots,t_d)\in\N_0^d:\,C_1^{\top}\bar{t}_1+\ldots+C_d^{\top}\bar{t}_d = \vec{0}\in\F_b^n\right\}.
\end{align*}

The following important fact is \cite[Remark 1]{D07}.

\begin{lem}\label{varrhogreater}
 $\P_n^b$ is an order $\sigma$ digital $(v,n,d)$-net over $\F_b$ if and only if $\varrho_\sigma(t)>\sigma n-v$ for all $t\in\Dn(\Cn)\setminus\{\vec{0}\}$.
\end{lem}

The following result is \cite[Lemma 2]{DP05}.

\begin{lem}\label{walsh_character}
Let $\P_n^b$ be an order $\sigma$ digital $(v,n,d)$-net over $\F_b$ with generating matrices $C_1,\ldots,C_d$. Then
 \[ \sum_{z\in\P_n^b} \wal_t(z) = \begin{cases} b^n & \text{if } t\in\Dn(\Cn),\\ 0 & \text{otherwise}.\end{cases} \]
\end{lem}

We consider the Walsh series expansion of the function $\chi_{[0,x)}$,
\begin{align}
\chi_{[0,x)}(y) = \sum_{\eta = 0}^\infty \hat{\chi}_{[0,x)}(\eta) \wal_\eta(y),
\end{align}
where for $\eta \in \N_0$ the $\eta$-th Walsh coefficient is given by
\[ \hat{\chi}_{[0,x)}(\eta) = \int_0^1 \chi_{[0,x)}(y) \overline{\wal_\eta(y)} \dint y = \int_0^x \overline{\wal_\eta(y)} \dint y. \]

\begin{lem}\label{lem_discr_walsh}
Let $\P_n^b$ be an order $\sigma$ digital $(v,n,d)$-net over $\F_b$ with generating matrices $C_1,\ldots,C_d$. Then
\[ D_{\P_n^b}(x) = \sum_{t\in\Dn(\Cn)\setminus\{\vec{0}\}}\hat{\chi}_{[0,x)}(t). \]
\end{lem}

\begin{proof}
 For $t = (t_1,\ldots,t_d)\in\N_0^d$ and $x = (x_1,\ldots,x_d)\in[0,1)^d$, we have
\[ \hat{\chi}_{[0,x)}(t) = \hat{\chi}_{[0,x_1)}(t_1) \cdots \hat{\chi}_{[0,x_d)}(t_d). \]
Applying Lemma \ref{walsh_character} we get
\begin{align*}
D_{\P}(x) & = \frac{1}{b^n} \sum_{z \in \P_n^b} \sum_{t_1, \ldots, t_d = 0}^{\infty} \hat{\chi}_{[0,x)}(t) \wal_t(z) - \hat{\chi}_{[0,x)}((0, \ldots, 0)) \\
          & = \sum_{\substack{t_1, \ldots, t_d = 0 \\ (t_1, \ldots t_d) \neq (0, \ldots, 0)}}^{\infty} \hat{\chi}_{[0,x)}(t) \frac{1}{b^n} \sum_{z \in \P} \wal_t(z) \\
          & = \sum_{t \in \Dn(\Cn)\setminus\{\vec{0}\}} \hat{\chi}_{[0,x)}(t).
\end{align*}
\end{proof}

Order $\sigma$ digital $(v,n,d)$-nets can be constructed from order $1$ digital $(w,n,\sigma d)$-nets using a method called digit interlacing (see \cite{DP14b} and \cite{D14} for details and examples). Constructions of order $1$ digital nets are well known. A good quality parameter $v$ that does not depend on $n$ can be obtained.

\section{Proofs of the results}
For two sequences $a_n$ and $b_n$  we will write $a_n\preceq b_n$ if there exists a constant $c>0$ such that $a_n\leq c\,b_n$ for all $n$. For $t > 0$ with $b$-adic expansion $t = \tau_0 + \tau_1 b + \ldots + \tau_{\varrho_1(t) - 1} b^{\varrho_1(t) - 1}$, we put $t = t' + \tau_{\varrho_1(t) - 1} b^{\varrho_1(t) - 1}$.

We start with two easy facts. For the proof of the first one see e.~g. \cite[Proof of Lemma 16.26]{DP10}.
\begin{lem}\label{index_dim_red}
Let $r\in\N_0$ and $s\in\N$. Then
\[ \#\{(a_1,\ldots,a_s)\in\N_0^s:\; a_1 + \ldots + a_s = r\} \leq (r + 1)^{s-1}. \]
\end{lem}

\begin{lem}\label{index_dim_red_log}
 Let $K\in\N$, $A>1$ and $q,s>0$. Then we have
\[ \sum_{r = 0}^{K-1} A^r (K-r)^q r^s \preceq A^K\,K^s, \]
 where the constant is independet of $K$.
\end{lem}

\begin{proof}
We have
 \[ \sum_{r = 0}^{K-1} A^r (K-r)^q r^s \leq A^K\,K^s\sum_{r=0}^{K-1} A^{r-K} (K-r)^q = A^K\,K^s\sum_{r=1}^{K} A^{-r} r^q \preceq A^K\,K^s. \]
\end{proof}

The following result is \cite[Lemma 5.1]{M13b}.
\begin{lem}\label{lem_haar_coeff_vol}
Let $f(x) = x_1\cdot\ldots\cdot x_d$ for $x=(x_1,\ldots,x_d)\in [0,1)^d$. Let $j\in\N_{-1}^d,\,m\in\D_j,l\in\B_j$. Then $|\langle f,h_{j,m,l}\rangle|\preceq b^{-2|j|_+}$.
\end{lem}

The following result is \cite[Lemma 5.2]{M13b}.
\begin{lem} \label{lem_haar_coeff_number}
Let $z = (z_1,\ldots,z_d) \in [0,1)^d$ and $g(x) = \chi_{[0,x)}(z)$ for $x = (x_1, \ldots, x_d) \in [0,1)^d$. Let $j\in\N_{-1}^d,\,m\in\D_j,l\in\B_j$. Then $\langle g,h_{j,m,l}\rangle = 0$ if $z$ is not contained in the interior of the $b$-adic interval $I_{j,m}$. If $z$ is contained in the interior of $I_{j,m}$ then $|\langle g,h_{j,m,l}\rangle|\preceq b^{-|j|_+}$.
\end{lem}

The following result is \cite[Lemma 5.9]{M13b}.
\begin{lem} \label{lem_haar_walsh_scalar}
Let $j\in\N_{-1}^d,\,m\in\D_j,\,l\in\B_j$ and $\alpha\in\N_0^d$. Then
\[ |\langle h_{j,m,l},\wal_\alpha\rangle| \preceq b^{-|j|_+}. \]
If $\varrho_1(\alpha_i) \neq j_i+1$ for some $1\leq i\leq d$ then
\[ \langle h_{j,m,l},\wal_\alpha\rangle = 0. \]
\end{lem}

The following result is \cite[Lemma 5.10]{M13b}.
\begin{lem} \label{lem_scalar_chiwalsh_walsh}
Let $t,\alpha\in\N_0$. Then
\[ |\langle\hat{\chi}_{[0,\cdot)}(t),\wal_\alpha\rangle|\preceq b^{-\max(\varrho_1(t),\varrho_1(\alpha))}. \]
If $\alpha\neq t'$ and $\alpha\neq t$ and $\alpha'\neq t$ then
\[ \langle\hat{\chi}_{[0,\cdot)}(t),\wal_\alpha\rangle=0. \]
\end{lem}

The following result is a modified version of \cite[Lemma 6]{DP14a}.

\begin{lem}\label{card_such_t}
 Let $C_1,\ldots,C_d\in\F_b^{s\times n}$ generate an order $1$ digital $(v,n,d)$-net over $\F_b$. Let $\lambda_1, \ldots, \lambda_d, \gamma_1, \ldots, \gamma_d\in\N_0$. Let $\omega_{\gamma_1,\ldots,\gamma_d}^{\lambda_1,\ldots,\lambda_d}(\Cn)$ denote the cardinality of such $t\in\Dn(\Cn)$ with $\varrho_1(t_i) = \gamma_i$ for all $1\leq i\leq d$ that either $\gamma_i\leq \lambda_i$ or $\varrho_1(t_i')=\lambda_i$. If $\lambda_1,\ldots,\lambda_d\leq s$ then
\[ \omega_{\gamma_1,\ldots,\gamma_d}^{\lambda_1,\ldots,\lambda_d}(\Cn) \leq (b-1)^d\,b^{\left(\min(\lambda_1,\gamma_1-1)+\ldots+\min(\lambda_d,\gamma_d-1)-n+v\right)_+}. \]
\end{lem}

\begin{proof}
 Let $t=(t_1,\ldots,t_d)\in\Dn(\Cn)$ with $\varrho_1(t_i) = \gamma_i$ for all $1\leq i\leq d$ and either $\gamma_i\leq \lambda_i$ or $\varrho_1(t_i')=\lambda_i$. Let $t_i$ have $b$-adic expansion $t_i = \tau_{i,0} + \tau_{i,1} b + \tau_{i,2} b^2 + \ldots$. Let $C_i=(c_{i,1},\ldots,c_{i,s})^{\top}$, put $\lambda^*_i=\min(\lambda_i,\gamma_i-1)$ and $c_{i,\gamma_i}=(0,\ldots,0)$ if $\gamma_i>s$, $1\leq i\leq d$. Then we have
\begin{align}
 & c_{1,1}^{\top}\tau_{1,0}+\ldots+c_{1,\lambda^*_1}^{\top}\tau_{1,\lambda^*_1-1}+c_{1,\gamma_1}^{\top}\tau_{1,\gamma_1-1}+\notag\\
 & \vdots\label{system}\\
+ & c_{d,1}^{\top}\tau_{d,0}+\ldots+c_{d,\lambda^*_d}^{\top}\tau_{d,\lambda^*_d-1}+c_{d,\gamma_d}^{\top}\tau_{d,\gamma_d-1} = (0\ldots,0)^{\top}\in\F_b^n\notag.
\end{align}
We put
\[ A = (c_{1,1}^{\top},\ldots,c_{1,\lambda^*_1}^{\top},\ldots,c_{d,1}^{\top},\ldots,c_{d,\lambda^*_d}^{\top})\in\F_b^{n\times(\lambda^*_1+\ldots+\lambda^*_d)}, \]
\[ y = (\tau_{1,0},\ldots,\tau_{1,\lambda^*_1-1},\ldots,\tau_{d,0},\ldots,\tau_{d,\lambda^*_d-1})^{\top}\in\F_b^{(\lambda^*_1+\ldots+\lambda^*_d)\times 1} \]
and
\[ w = -c_{1,\gamma_1}^{\top}\tau_{1,\gamma_1-1}-\ldots-c_{d,\gamma_d}^{\top}\tau_{d,\gamma_d-1}\in\F_b^{n\times 1}. \]
Then \eqref{system} corresponds to $Ay=w$ and we have
\[ \omega_{\gamma_1,\ldots,\gamma_d}^{\lambda_1,\ldots,\lambda_d}(\Cn) = \#\{(y,w)\in\F_b^{\lambda^*_1+\ldots+\lambda^*_d}\times\F_b^n:\,Ay=w\}. \]
Since $C_1,\ldots,C_d$ generate an order $1$ digital $(v,n,d)$-net, the rank of $A$ is $\lambda^*_1+\ldots+\lambda^*_d$ if $\lambda^*_1+\ldots+\lambda^*_d\leq n-v$. In this case the solution space of the homogeneous system $Ay=(0,\ldots,0)$ has dimension $0$. If $\lambda^*_1+\ldots+\lambda^*_d>n-v$ then $\rank(A)\geq n-v$ and the dimension of the solution space of the homogeneous system is $\lambda^*_1+\ldots+\lambda^*_d-\rank(A)\leq\lambda_1+\ldots+\lambda_d-n+v$. This means that for a given $w$ the system $Ay=w$ has at most $1$ solution if $\lambda^*_1+\ldots+\lambda^*_d\leq n-v$ and at most $b^{\lambda^*_1+\ldots+\lambda^*_d-n+v}$ solutions otherwise. Finally, there are $(b-1)^d$ possible choices for $w$ since none of the numbers $\tau_{1,\gamma_1-1},\ldots,\tau_{d,\gamma_d-1}$ can be $0$.
\end{proof}

We point out that the condition $\lambda_1,\ldots,\lambda_d\leq s$ is not necessary. It just reduces the technicalities but the results would be the same without it. One would have to define $\lambda^{**}_i=\min(\lambda_i*,s)$ and in the case where $\lambda_i* >s$ we would get an additional factor $b^{\lambda_i*-s}$ compensating the restriction.

\begin{lem}\label{prp_ord1}
 Let $\P_n^b$ be an order 1 digital $(v,n,d)$-net over $\F_b$. Let $j\in\N_{-1}^d,\,m\in\D_j,\,l\in\B_j$.
\begin{enumerate}[(i)]
 \item If $|j|_+ \geq n-v$ then $|\langle D_{\P_n^b},h_{j,m,l}\rangle|\preceq b^{-|j|_+ -n+v}$ and $|\langle D_{\P_n^b},h_{j,m,l}\rangle|\preceq b^{-2|j|_+}$ for all but at most $b^n$ values of $m$.\label{prp_ord1_part1}
 \item If $|j|_+ < n-v$ then $|\langle D_{\P_n^b},h_{j,m,l}\rangle|\preceq b^{-|j|_+ -n+v}\left(n-v-|j|_+\right)^{d-1}$.\label{prp_ord1_part2}   
\end{enumerate}                                                                                            
\end{lem}

\begin{proof}
For \eqref{prp_ord1_part1}, let $|j|_+ \geq n-v$. Since $\P_n^b$ contains exactly $b^n$ points, there are no more than $b^n$ such $m$ for which $I_{j,m}$ contains a point of $\P_n^b$ meaning that at least all but $b^n$ intervals contain no points at all. Thus the second statement follows from Lemmas \ref{lem_haar_coeff_vol} and \ref{lem_haar_coeff_number}. The remaining intervals contain at most $b^v$ points of $\P_n^b$ (Lemma \ref{nets_distribution}) so the first statement follows from Lemmas \ref{lem_haar_coeff_vol} and \ref{lem_haar_coeff_number}.

We now prove \eqref{prp_ord1_part2}. Let $|j|_+ < n-v$ and $m\in\D_j,\,l\in\B_j$. The function $h_{j,m,l}$ can be written (Lemma \ref{Walsh_basis}) as
\[ h_{j,m,l} = \sum_{\alpha\in\N_0^d} \langle h_{j,m,l},\wal_{\alpha}\rangle\wal_\alpha. \]

We apply Lemmas \ref{lem_discr_walsh}, \ref{lem_haar_walsh_scalar} and \ref{lem_scalar_chiwalsh_walsh} and get
\begin{align}
&|\langle D_{\P_n^b},h_{j,m,l}\rangle| = \left|\left\langle \sum_{t\in\Dn(\Cn)\setminus\{\vec{0}\}}\hat{\chi}_{[0,\cdot)}(t),\sum_{\alpha\in\N_0^d}\left\langle h_{j,m,l},\wal_{\alpha}\right\rangle\wal_{\alpha}\right\rangle\right|\notag\\
&\leq\sum_{t\in\Dn(\Cn)\setminus\{\vec{0}\}}\sum_{\alpha \in \N_0^d}\left|\left\langle\hat{\chi}_{[0,\cdot)}(t),\wal_{\alpha}\right\rangle\right|\left|\left\langle h_{j,m,l},\wal_{\alpha}\right\rangle\right|\notag\\
&\leq b^{-|j|_+}\sum_{t\in\Dn(\Cn)\setminus\{\vec{0}\}}\sum_{\substack{\alpha \in \N_0^d\\ \varrho_1(\alpha_i)=j_i+1\\1\leq i\leq d}}\left|\left\langle\hat{\chi}_{[0,\cdot)}(t),\wal_{\alpha}\right\rangle\right|\notag\\
&\leq b^{-|j|_+}\sum_{t\in\Dn(\Cn)\setminus\{\vec{0}\}}\sum_{\substack{\alpha \in \N_0^d\\\alpha_i=t_i'\,\vee\,\alpha_i=t_i\,\vee\,\alpha_i'=t_i\\ \varrho_1(\alpha_i)=j_i+1,\,1\leq i\leq d}}b^{-\max(\varrho_1(\alpha_1),\varrho_1(t_1))-\ldots-\max(\varrho_1(\alpha_1),\varrho_1(t_d))}\notag\\
&= b^{-|j|_+}\sum_{\substack{t\in\Dn(\Cn)\setminus\{\vec{0}\}\\ \varrho_1(t_i)\leq j_i+1\,\vee\,\varrho_1(t_i')=j_i+1\\1\leq i\leq d}}b^{-\max(j_1+1,\varrho_1(t_1))-\ldots-\max(j_d+1,\varrho_1(t_d))}\notag\\
&= b^{-|j|_+}\sum_{\gamma_1,\ldots,\gamma_d=0}^\infty b^{-\max(j_1+1,\gamma_1)-\ldots-\max(j_d+1,\gamma_d)}\,\omega_{\gamma_1,\ldots,\gamma_d}^{j_1+1,\ldots,j_d+1}(\Cn)\label{calc_gamma}\\
&= b^{-|j|_+}\sum_{\substack{\gamma_1,\ldots,\gamma_d=0\\ \gamma_1+\ldots+\gamma_d>n-v}}^\infty b^{-\max(j_1+1,\gamma_1)-\ldots-\max(j_d+1,\gamma_d)}\,\omega_{\gamma_1,\ldots,\gamma_d}^{j_1+1,\ldots,j_d+1}(\Cn) +\notag \\
& \qquad + b^{-|j|_+}\sum_{\substack{\gamma_1,\ldots,\gamma_d=0\\ \gamma_1+\ldots+\gamma_d\leq n-v}}^\infty b^{-\max(j_1+1,\gamma_1)-\ldots-\max(j_d+1,\gamma_d)}\,\omega_{\gamma_1,\ldots,\gamma_d}^{j_1+1,\ldots,j_d+1}(\Cn)\notag.
\end{align}
By Lemma \ref{card_such_t} we get
\[ \omega_{\gamma_1,\ldots,\gamma_d}^{j_1+1,\ldots,j_d+1}(\Cn) \leq (b-1)^d\,b^d \]
since $j_1+1,\ldots,j_d+1\leq n-v\leq s$ and $j_1+1+\ldots+j_d+1\leq  |j|_+ +d< n-v+d$. We apply this only to the first sum incorporating this term into the constant. The second sum vanishes. To see that we recall that $\varrho_1(t)>n-v$ for all $t\in\Dn(\Cn)\setminus\{\vec{0}\}$. This means that $\omega_{\gamma_1,\ldots,\gamma_d}^{j_1+1,\ldots,j_d+1}(\Cn)=0$ whenever $\gamma_1+\ldots+\gamma_d\leq n-v$ since $\varrho_1(t) = \gamma_1+\ldots+\gamma_d$ and the second sum vanishes. For any $I\subset\{1,\ldots,d\}$ let $I^c=\{1,\ldots,d\}\setminus I$. So far we have
\begin{align*}
&|\langle D_{\P_n^b},h_{j,m,l}\rangle| \preceq b^{-|j|_+}\sum_{\substack{\gamma_1,\ldots,\gamma_d=0\\ \gamma_1+\ldots+\gamma_d>n-v}}^\infty b^{-\max(j_1+1,\gamma_1)-\ldots-\max(j_d+1,\gamma_d)}\\
&=b^{-|j|_+}\sum_{I\subsetneq\{1,\ldots,d\}}b^{-\sum\limits_{\kappa_1\in I}(j_{\kappa_1}+1)}\underset{\gamma_1+\ldots+\gamma_d\geq\max\left(n-v+1,\sum\limits_{\kappa_2\in I^c}(j_{\kappa_2}+1)\right)}{\sum_{\substack{0\leq\gamma_{i_1}\leq j_{i_1}\\i_1\in I}}\;\sum_{\substack{\gamma_{i_2}\geq j_{i_2}+1\\i_2\in I^c}}} b^{-\sum\limits_{\kappa_2\in I^c}\gamma_{\kappa_2}}.\\
\end{align*}
The case where $I = \{1,\ldots,d\}$ is not possible (therefore excluded) because $\gamma_i\leq j_i$ for all $1\leq i\leq d$ contradicts the condition $\gamma_1+\ldots+\gamma_d>n-v$ since $j_1 + \ldots + j_d < n-v$. We perform an index shift to get
\begin{align*}
&|\langle D_{\P_n^b},h_{j,m,l}\rangle| \preceq b^{-|j|_+}\sum_{I\subsetneq\{1,\ldots,d\}}b^{-\sum\limits_{\kappa_1\in I}(j_{\kappa_1}+1)-\sum\limits_{\kappa_2\in I^c}(j_{\kappa_2}+1)}\ldots\\&\hspace{2.5cm}\ldots\sum_{\substack{0\leq\gamma_{i_1}\leq j_{i_1}\\i_1\in I}}\;\sum_{\substack{\gamma_{i_2}\geq 0,\,i_2\in I^c\\ \sum\limits_{\kappa_2\in I^c}\gamma_{\kappa_2}\geq\left(n-v-\sum\limits_{\kappa_1\in I}\gamma_{\kappa_1}-\sum\limits_{\kappa_2\in I^c}(j_{\kappa_2}+1)+1\right)_+}} b^{-\sum\limits_{\kappa_2\in I^c}\gamma_{\kappa_2}}.
\end{align*}
We apply Lemma \ref{index_dim_red} to obtain
\begin{align*}
&\leq b^{-|j|_+}\sum_{I\subsetneq\{1,\ldots,d\}}b^{-\sum\limits_{\kappa_1\in I}(j_{\kappa_1}+1)-\sum\limits_{\kappa_2\in I^c}(j_{\kappa_2}+1)}\ldots\\&\hspace{2cm}\ldots\sum_{\substack{0\leq\gamma_{i_1}\leq j_{i_1}\\i_1\in I}}\;\sum_{r=\left(n-v-\sum\limits_{\kappa_1\in I}\gamma_{\kappa_1}-\sum\limits_{\kappa_2\in I^c}(j_{\kappa_2}+1)+1\right)_+}^\infty b^{-r}(r+1)^{d-1-\#I}\\
&\leq b^{-|j|_+}\sum_{I\subsetneq\{1,\ldots,d\}}b^{-\sum\limits_{\kappa_1\in I}(j_{\kappa_1}+1)-\sum\limits_{\kappa_2\in I^c}(j_{\kappa_2}+1)}\sum_{\substack{0\leq\gamma_{i_1}\leq j_{i_1}\\i_1\in I}}\;b^{-n+v+\sum\limits_{\kappa_1\in I}\gamma_{\kappa_1}+\sum\limits_{\kappa_2\in I^c}(j_{\kappa_2}+1)}\\&\hspace{4.5cm}\times\left(n-v-\sum\limits_{\kappa_1\in I}\gamma_{\kappa_1}-\sum\limits_{\kappa_2\in I^c}(j_{\kappa_2}+1)+1\right)_+^{d-1-\#I}\\
&\leq b^{-|j|_+ -n+v}\sum_{I\subsetneq\{1,\ldots,d\}}b^{-\sum\limits_{\kappa_1\in I}(j_{\kappa_1}+1)}\sum_{\substack{0\leq\gamma_{i_1}\leq j_{i_1}\\i_1\in I}}b^{\sum\limits_{\kappa_1\in I}\gamma_{\kappa_1}}\\&\hspace{4.5cm}\times\left(n-v-\sum\limits_{\kappa_1\in I}\gamma_{\kappa_1}-\sum\limits_{\kappa_2\in I^c}(j_{\kappa_2}+1)+1\right)_+^{d-1}\\
&\leq b^{-|j|_+ -n+v}\sum_{I\subsetneq\{1,\ldots,d\}}b^{-\sum\limits_{\kappa_1\in I}(j_{\kappa_1}+1)}b^{\sum\limits_{\kappa_1\in I}(j_{\kappa_1}+1)}\\&\hspace{4cm}\times\left(n-v-\sum\limits_{\kappa_1\in I}(j_{\kappa_1}+1)-\sum\limits_{\kappa_2\in I^c}(j_{\kappa_2}+1)+1\right)_+^{d-1}\\
&\preceq b^{-|j|_+ -n+v}\left(n-v-|j|_+\right)^{d-1}.
\end{align*}
\end{proof}

\begin{lem}\label{prp_ord2}
 Let $\P_n^b$ be an order $2$ digital $(v,n,d)$-net over $\F_b$. Let $j\in\N_{-1}^d,\,m\in\D_j,\,l\in\B_j$.
\begin{enumerate}[(i)]
 \item If $|j|_+\geq n-\lceil v/2\rceil$ then $|\langle D_{\P_n^b},h_{j,m,l}\rangle|\preceq b^{-|j|_+-n+v/2}$ and $|\langle D_{\P_n^b},h_{j,m,l}\rangle|\preceq b^{-2|j|_+}$ for all but $b^n$ values of $m$.\label{prp_ord2_part1}
 \item If $|j|_+< n-\lceil v/2\rceil$ then $|\langle D_{\P_n^b},h_{j,m,l}\rangle|\preceq b^{-2n+v}\left(2n-v-2|j|_+\right)^{d-1}$.\label{prp_ord2_part2}
\end{enumerate}                                                                                                 
\end{lem}

\begin{proof}
According to Lemma \ref{lem_digital_nets_sigma_v}, $\P_n^b$ is an order $1$ digital $(\lceil v/2\rceil,n,d)$-net. Hence \eqref{prp_ord2_part1} follows from Lemma \ref{prp_ord1}.

We now prove \eqref{prp_ord2_part2}. Let $|j|_+< n-\lceil v/2\rceil$ and $m\in\D_j,\,l\in\B_j$. We start at \eqref{calc_gamma} so we have
\begin{align*}
&|\langle D_{\P_n^b},h_{j,m,l}\rangle| \\
&\preceq b^{-|j|_+}\sum_{\substack{\gamma_1,\ldots,\gamma_d=0\\ \sum_{i=1}^d\gamma_i+\min(\gamma_i,j_i+1)>2n-v}}^\infty b^{-\max(j_1+1,\gamma_1)-\ldots-\max(j_d+1,\gamma_d)}\,\omega_{\gamma_1,\ldots,\gamma_d}^{j_1+1,\ldots,j_d+1}(\Cn) \\
& \; + b^{-|j|_+}\sum_{\substack{\gamma_1,\ldots,\gamma_d=0\\ \sum_{i=1}^d\gamma_i+\min(\gamma_i,j_i+1)\leq 2n-v}}^\infty b^{-\max(j_1+1,\gamma_1)-\ldots-\max(j_d+1,\gamma_d)}\,\omega_{\gamma_1,\ldots,\gamma_d}^{j_1+1,\ldots,j_d+1}(\Cn).
\end{align*}
We argue similarly to the proof of Lemma \ref{prp_ord1}, incorporating the term $\omega_{\gamma_1,\ldots,\gamma_d}^{j_1+1,\ldots,j_d+1}(\Cn)$ in the first sum into the constant and seeing that the second sum vanishes. To see that the second sum vanishes we recall that we have $\varrho_2(t)>2n-v$ for all $t\in\Dn(\Cn)$. This means that $\omega_{\gamma_1,\ldots,\gamma_d}^{j_1+1,\ldots,j_d+1}(\Cn)=0$ whenever $\gamma_1+\min(\gamma_1,j_1+1)+\ldots+\gamma_d+\min(\gamma_d,j_d+1)\leq 2n-v$ because $\varrho_2(t)\leq\gamma_1+\min(\gamma_1,j_1+1)+\ldots+\gamma_d+\min(\gamma_d,j_d+1)$ since $\varrho_1(t_i) = \gamma_i$ and $\varrho_1(t_i') = j_i + 1$ if $\gamma_i>j_i + 1$ for all $1\leq i\leq d$. With the same arguments as in the proof of Lemma \ref{prp_ord1} we obtain
\begin{align*}
&|\langle D_{\P_n^b},h_{j,m,l}\rangle| \preceq b^{-|j|_+}\sum_{\substack{\gamma_1,\ldots,\gamma_d=0\\ \sum_{i=1}^d\gamma_i+\min(\gamma_i,j_i+1)>2n-v}}^\infty b^{-\max(j_1+1,\gamma_1)-\ldots-\max(j_d+1,\gamma_d)}\\
&=b^{-|j|_+}\sum_{I\subsetneq\{1,\ldots,d\}}b^{-\sum\limits_{\kappa_1\in I}(j_{\kappa_1}+1)}\ldots\\&\hspace{2.5cm}\ldots\underset{2\sum\limits_{\kappa_1\in I}\gamma_{\kappa_1}+\sum\limits_{\kappa_2\in I^c}(\gamma_{\kappa_2}+j_{\kappa_2}+1)\geq\max\left(2n-v+1,2\sum\limits_{\kappa_2\in I^c}(j_{\kappa_2}+1)\right)}{\sum_{\substack{0\leq\gamma_{i_1}\leq j_{i_1}\\i_1\in I}}\;\sum_{\substack{\gamma_{i_2}\geq j_{i_2}+1\\i_2\in I^c}}} b^{-\sum\limits_{\kappa_2\in I^c}\gamma_{\kappa_2}}\\
&=b^{-|j|_+}\sum_{I\subsetneq\{1,\ldots,d\}}b^{-\sum\limits_{\kappa_1\in I}(j_{\kappa_1}+1)-\sum\limits_{\kappa_2\in I^c}(j_{\kappa_2}+1)}\ldots\\&\hspace{2cm}\ldots\sum_{\substack{0\leq\gamma_{i_1}\leq j_{i_1}\\i_1\in I}}\;\sum_{\substack{\gamma_{i_2}\geq 0,\,i_2\in I^c\\ \sum\limits_{\kappa_2\in I^c}\gamma_{\kappa_2}\geq\left(2n-v-2\sum\limits_{\kappa_1\in I}\gamma_{\kappa_1}-2\sum\limits_{\kappa_2\in I^c}(j_{\kappa_2}+1)+1\right)_+}} b^{-\sum\limits_{\kappa_2\in I^c}\gamma_{\kappa_2}}\\
&\leq b^{-|j|_+}\sum_{I\subsetneq\{1,\ldots,d\}}b^{-\sum\limits_{\kappa_1\in I}(j_{\kappa_1}+1)-\sum\limits_{\kappa_2\in I^c}(j_{\kappa_2}+1)}\ldots\\&\hspace{2cm}\ldots\sum_{\substack{0\leq\gamma_{i_1}\leq j_{i_1}\\i_1\in I}}\;\sum_{r=\left(2n-v-2\sum\limits_{\kappa_1\in I}\gamma_{\kappa_1}-2\sum\limits_{\kappa_2\in I^c}(j_{\kappa_2}+1)+1\right)_+}^\infty b^{-r}(r+1)^{d-1-\#I}
\end{align*}
where we applied Lemma \ref{index_dim_red} and several index shifts. The case $I = \{1,\ldots,d\}$ contradicts the condition $\varrho_2(t)>2n-v$ since $\varrho_2(t) < 2j_1 + \ldots + 2j_d < 2n-2v \leq 2n-v$. We continue the calculation
\begin{align*}
&|\langle D_{\P_n^b},h_{j,m,l}\rangle| \\
& \preceq b^{-|j|_+}\sum_{I\subsetneq\{1,\ldots,d\}}b^{-\sum\limits_{\kappa_1\in I}(j_{\kappa_1}+1)-\sum\limits_{\kappa_2\in I^c}(j_{\kappa_2}+1)}\sum_{\substack{0\leq\gamma_{i_1}\leq j_{i_1}\\i_1\in I}}\;b^{-2n+v+2\sum\limits_{\kappa_1\in I}\gamma_{\kappa_1}+2\sum\limits_{\kappa_2\in I^c}(j_{\kappa_2}+1)}\\&\hspace{3.5cm}\times\left(2n-v-2\sum\limits_{\kappa_1\in I}\gamma_{\kappa_1}-2\sum\limits_{\kappa_2\in I^c}(j_{\kappa_2}+1)+1\right)^{d-1-\#I}\\
&\leq b^{-|j|_+ -2n+v}\sum_{I\subsetneq\{1,\ldots,d\}}b^{-\sum\limits_{\kappa_1\in I}(j_{\kappa_1}+1)+\sum\limits_{\kappa_2\in I^c}(j_{\kappa_2}+1)}\sum_{\substack{0\leq\gamma_{i_1}\leq j_{i_1}\\i_1\in I}}b^{2\sum\limits_{\kappa_1\in I}\gamma_{\kappa_1}}\\&\hspace{4cm}\times\left(2n-v-2\sum\limits_{\kappa_1\in I}\gamma_{\kappa_1}-2\sum\limits_{\kappa_2\in I^c}(j_{\kappa_2}+1)+1\right)^{d-1}\\
&\leq b^{-|j|_+ -2n+v}\sum_{I\subsetneq\{1,\ldots,d\}}b^{\sum\limits_{\kappa_1\in I}(j_{\kappa_1}+1)+\sum\limits_{\kappa_2\in I^c}(j_{\kappa_2}+1)}\\&\hspace{3.5cm}\times\left(2n-v-2\sum\limits_{\kappa_1\in I}(j_{\kappa_1}+1)-2\sum\limits_{\kappa_2\in I^c}(j_{\kappa_2}+1)+1\right)^{d-1}\\
&\preceq b^{-2n+v}\left(2n-v-2|j|_+\right)^{d-1}.
\end{align*}
\end{proof}

We are now ready to prove the theorems.

\begin{proof}[Proof of Theorem \ref{main_thm_L2}]
Let $\P_n^b$ be an order $2$ digital $(v,n,d)$-net over $\F_b$. We apply Theorem \ref{L2Haar} and prove
\begin{align} \label{main_thm_L2_v_dep}
 \sum_{j\in\N_{-1}^d}b^{|j|_+}\sum_{m\in\D_j,\,l\in\B_j}|\langle D_{\P_n^b},h_{j,m,l}\rangle|^2\preceq b^{-2n+v}\,n^{d-1}\,v \preceq b^{-2n}\,n^{d-1}.
\end{align}
We recall that $\#\D_j=b^{|j|_+},\,\#\B_j=b-1$. We split the sum in $j$ into three parts and apply Lemma \ref{prp_ord2} \eqref{prp_ord2_part2} and Lemma \ref{index_dim_red_log} to get
\begin{align*}
 &\sum_{\substack{j\in\N_{-1}^d\\|j|_+< n-\lceil v/2\rceil}}b^{|j|_+}\sum_{m\in\D_j,\,l\in\B_j}|\langle D_{\P_n^b},h_{j,m,l}\rangle|^2\\
 &\preceq  \sum_{\substack{j\in\N_{-1}^d\\|j|_+< n-\lceil v/2\rceil}}b^{|j|_+}\,b^{|j|_+}\,b^{-4n+2v}\left(2n-v-2|j|_+\right)^{2(d-1)}\\
 &\leq b^{-4n+2v}\sum_{\kappa=0}^{n-v/2-1}b^{2\kappa}\left(2n-v-2\kappa\right)^{2(d-1)}(\kappa+1)^{d-1}\\
 &\leq b^{-4n+2v}\,b^{2n-v}\,\left(2n-v-2n+v+2\right)^{2(d-1)}(n-v/2)^{d-1}\\
 &\preceq b^{-2n+v}\,n^{d-1}
\end{align*}
for big intervals. We also consider middle sized and small intervals. In the case of small intervals ($|j|_+\geq n$) there are at most $b^n$ intervals containing a point of $\P_n^b$ while in the case where $n> |j|_+\geq n$ there are even less namely at most $b^{|j|_+}$. We apply Lemma \ref{prp_ord2} \eqref{prp_ord2_part1}
\begin{align*}
 &\sum_{\substack{j\in\N_{-1}^d\\n> |j|_+\geq n-\lceil v/2\rceil}}b^{|j|_+}\sum_{m\in\D_j,\,l\in\B_j}|\langle D_{\P_n^b},h_{j,m,l}\rangle|^2\\
 &\preceq \sum_{\substack{j\in\N_{-1}^d\\n> |j|_+\geq n-\lceil v/2\rceil}}\,b^{|j|_+}\,b^{|j|_+}\,b^{-2|j|_+ -2n+v}\\
 &\leq b^{-2n+v}\sum_{\kappa=n-\lceil v/2\rceil}^{n-1}(\kappa+1)^{d-1}\\
 &\preceq b^{-2n+v}\,n^{d-1}\,v
\end{align*}
for medium sized intervals and
\begin{align*}
 &\sum_{\substack{j\in\N_{-1}^d\\|j|_+\geq n}}b^{|j|_+}\sum_{m\in\D_j,\,l\in\B_j}|\langle D_{\P_n^b},h_{j,m,l}\rangle|^2\\
 &\preceq \sum_{\substack{j\in\N_{-1}^d\\|j|_+\geq n}}\,b^{|j|_+}\,b^n\,b^{-2|j|_+ -2n+v} + \sum_{\substack{j\in\N_{-1}^d\\|j|_+\geq n}}\,b^{|j|_+}\,(b^{|j|_+}-b^n)\,b^{-4|j|_+}\\
 &\leq b^{-n+v}\sum_{\kappa=n}^{\infty}b^{-\kappa}\,(\kappa+1)^{d-1} + \sum_{\kappa=n}^{\infty}b^{-2\kappa}\,(\kappa+1)^{d-1}\\
 &\preceq b^{-2n+v}\,n^{d-1}
\end{align*}
for small intervals.
\end{proof}

\begin{proof}[Proof of Theorem \ref{main_thm_SpqrB_ord1}]
 Let $D_{\P_n^b}$ be an order $1$ digital $(v,n,d)$-net over $\F_b$. We apply Theorem \ref{SpqrBHaar} and prove
\begin{align} \label{main_thm_SpqrB_ord1_v_dep}
 \sum_{j\in\N_{-1}^d} b^{|j|_+(r-1/p+1)q}\left(\sum_{m\in\D_j,\,l\in\B_j}|\langle D_{\P_n^b},h_{j,m,l}\rangle|^p\right)^{q/p}&\preceq b^{n(r-1)q}\,n^{(d-1)}\,b^{vq} \\&\preceq b^{n(r-1)q}\,n^{(d-1)}\notag.
\end{align}
We recall that $\#\D_j=b^{|j|_+},\,\#\B_j=b-1$. We split the sum in $j$ in three parts and apply Minkowski's inequality, Lemma \ref{prp_ord1} \eqref{prp_ord1_part2} and Lemma \ref{index_dim_red_log} to get
\begin{align*}
 &\sum_{\substack{j\in\N_{-1}^d\\|j|_+ < n-v}}b^{|j|_+(r-1/p+1)q}\left(\sum_{m\in\D_j,\,l\in\B_j}|\langle D_{\P_n^b},h_{j,m,l}\rangle|^p\right)^{q/p}\\
 &\preceq  \sum_{\substack{j\in\N_{-1}^d\\|j|_+ < n-v}} b^{|j|_+(r-1/p+1)q}\,b^{|j|_+ q/p}\,b^{(-|j|_+ -n+v)q}\left(n-v-|j|_+\right)^{(d-1)q}\\
 &\leq b^{(-n+v)q}\sum_{\kappa=0}^{n-v-1}b^{\kappa rq}\left(n-v-\kappa\right)^{(d-1)q}(\kappa+1)^{d-1}\\
 &\leq b^{(-n+v)q}\,b^{(n-v)rq}\,(n-v+1)^{d-1}\\
 &\preceq b^{n(r-1)q}\,n^{d-1}\,b^{v(1-r)q}
\end{align*}
for big intervals. Again we differentiate between small intervals and middle sized intervals. We apply Lemma \ref{prp_ord1} \eqref{prp_ord1_part1}
\begin{align*}
 &\sum_{\substack{j\in\N_{-1}^d\\n> |j|_+\geq n-v}}b^{|j|_+(r-1/p+1)q}\left(\sum_{m\in\D_j,\,l\in\B_j}|\langle D_{\P_n^b},h_{j,m,l}\rangle|^p\right)^{q/p}\\
 &\preceq \sum_{\substack{j\in\N_{-1}^d\\n> |j|_+\geq n-v}}b^{|j|_+(r-1/p+1)q}\,b^{|j|_+ q/p}\,b^{(-|j|_+ -n+v)q}\\
 &\leq b^{(-n+v)q}\sum_{\kappa=n-v}^{n-1}b^{\kappa rq}(\kappa+1)^{d-1}\\
 &\preceq b^{(-n+v)q}\,b^{nrq}\,n^{d-1}\\
 &\leq b^{n(r-1)q}\,n^{(d-1)}\,b^{vq}
\end{align*}
for medium sized intervals and considering the range of $r$
\begin{align*}
 &\sum_{\substack{j\in\N_{-1}^d\\|j|_+\geq n}}b^{|j|_+(r-1/p+1)q}\left(\sum_{m\in\D_j,\,l\in\B_j}|\langle D_{\P_n^b},h_{j,m,l}\rangle|^p\right)^{q/p}\\
 &\preceq \sum_{\substack{j\in\N_{-1}^d\\|j|_+\geq n}}\,b^{|j|_+(r-1/p+1)q}\,b^{nq/p}\,b^{(-|j|_+ -n+v)q}\\ &\qquad+ \sum_{\substack{j\in\N_{-1}^d\\|j|_+\geq n}}\,b^{|j|_+(r-1/p+1)q}\,(b^{|j|_+}-b^n)^{q/p}\,b^{-2|j|_+ q}\\
 &\leq b^{nq/p}\,b^{(-n+v)q}\sum_{\kappa=n}^{\infty}b^{\kappa(r-1/p)q}\,(\kappa+1)^{d-1} + \sum_{\kappa=n}^{\infty}b^{\kappa(r-1)q}\,(\kappa+1)^{d-1}\\
 &\preceq b^{nq/p}\,b^{(-n+v)q}\,b^{n(r-1/p)q}n^{d-1} + b^{n(r-1)q}\,n^{d-1}\\
 &\preceq b^{n(r-1)q}\,n^{(d-1)}\,b^{vq}
\end{align*}
for small intervals.
\end{proof}

\begin{proof}[Proof of Theorem \ref{main_thm_SpqrB_ord2}]
 Let $D_{\P_n^b}$ be an order $2$ digital $(v,n,d)$-net over $\F_b$. The proof is similar to the proof of Theorem \ref{main_thm_SpqrB_ord1}. We apply Lemma \ref{prp_ord2} instead of \ref{prp_ord1} to get
\begin{align}
 &\sum_{\substack{j\in\N_{-1}^d\\|j|_+ < n-\lceil v/2\rceil}}b^{|j|_+(r-1/p+1)q}\left(\sum_{m\in\D_j,\,l\in\B_j}|\langle D_{\P_n^b},h_{j,m,l}\rangle|^p\right)^{q/p} \notag\\
 &\preceq  \sum_{\substack{j\in\N_{-1}^d\\|j|_+ < n-\lceil v/2\rceil}} b^{|j|_+(r-1/p+1)q}\,b^{|j|_+ q/p}\,b^{(-2n+v)q}\left(2n-v-2|j|_+\right)^{(d-1)q} \notag\\
 &\leq b^{(-2n+v)q}\sum_{\kappa=0}^{n-v/2-1}b^{\kappa (r+1)q}\left(2n-v-2\kappa\right)^{(d-1)q}(\kappa+1)^{d-1}\notag \\
 &\leq b^{(-2n+v)q}\,b^{(n-v/2)(r+1)q}\,(n-v/2+1)^{d-1} \notag\\
 &\preceq b^{n(r-1)q}\,n^{d-1}\,b^{v/2(1-r)q} \label{main_thm_SpqrB_ord2_v_dep} \\
 &\preceq b^{n(r-1)q}\,n^{d-1} \notag
\end{align}
and analogous results for the other subsums.
\end{proof}

\begin{proof}[Proof of Corollaries \ref{main_cor_SpqrF_ord1} and \ref{main_cor_SpqrF_ord2}]
 The results for the Triebel-Lizorkin spaces follow from the results for the Besov spaces. We apply Lemma \ref{lem_emb_BF}: there is a constant $c > 0$ such that
\[ \left\|D_{\P_n^b}|S_{p,q}^r F\right\| \leq c\, \left\|D_{\P_n^b}|S_{\max(p,q),q}^r B\right\| \]
and Corollary \ref{main_cor_SpqrF_ord1} follows from Theorem \ref{main_thm_SpqrB_ord1} and Corollary \ref{main_cor_SpqrF_ord2} from Theorem \ref{main_thm_SpqrB_ord2}.
\end{proof}

\begin{proof}[Proof of Corollaries \ref{main_cor_SprH_ord1} and \ref{main_cor_SprH_ord2}]
 We recall that $S_p^r H = S_{p,2}^r F$. Therefore Corollary \ref{main_cor_SprH_ord1} follows from Corollary \ref{main_cor_SpqrF_ord1} and Corollary \ref{main_cor_SprH_ord2} from Corollary \ref{main_cor_SpqrF_ord2} in the case $q=2$, respectively.
\end{proof}

\begin{proof}[Proof of Corollary \ref{main_thm_Lp}]
 We recall that $L_p = S_p^0 H$. Therefore the result follows from Corollary \ref{main_cor_SprH_ord2} in the case $r=0$.
\end{proof}

\addcontentsline{toc}{chapter}{References}

\end{document}